\documentclass[reqno]{amsart}
\usepackage{geometry}

\newgeometry{vmargin={15mm}, hmargin={37mm,37mm}}   

\usepackage[enableskew,vcentermath]{youngtab}
\usepackage{ytableau}

\usepackage{amsfonts}
\usepackage[enableskew,vcentermath]{youngtab}

\usepackage{grffile}
\usepackage{amsmath,amssymb,graphicx}
\usepackage{verbatim} 

\usepackage[latin1]{inputenc}
\usepackage{enumerate}
\usepackage{hyperref}
\usepackage{stackrel}
\usepackage{graphicx}
\DeclareMathOperator{\des}{des}
\DeclareMathOperator{\Des}{Des}
\DeclareMathOperator{\ides}{ides}

\DeclareMathOperator{\simp}{simp}
\DeclareMathOperator{\Simp}{Simp}
\DeclareMathOperator{\Av}{Av}

\begin{document}

\newtheorem{thm}{Theorem}[section]
\newtheorem{prop}[thm]{Proposition}
\newtheorem{formula}[thm]{Formula}
\newtheorem{lem}[thm]{Lemma}
\newtheorem{cla}[thm]{Claim}
\newtheorem{con}[thm]{Conjecture}
\newtheorem{cor}[thm]{Corollary}
\newtheorem{prb}[thm]{Problem}
\newtheorem{que}[thm]{Question}
\theoremstyle{definition}
\newtheorem{obs}[thm]{Observation}
\newtheorem{pres}[thm]{Presentation}
\newtheorem{fac}[thm]{Fact}
\newtheorem{rem}[thm]{Remark}
\newtheorem{exa}[thm]{Example}
\newtheorem{defn}[thm]{Definition}
\newtheorem{df}[thm]{Definition}
\newtheorem{alg}[thm]{Algorithm}


\newcommand{\sumlim}{\sum\limits}
\newcommand{\be}{\beta}
\newcommand{\dep}{{\rm dp}}
\newcommand{\Neg}{{\rm Neg}}
\newcommand{\nneg}{{\rm neg}}
\newcommand{\inv}{{\rm inv}}
\newcommand{\sgn}{{\rm sign}}

\title[On two-sided gamma-positivity for simple permutations]
{On two-sided gamma-positivity for simple permutations\\}


\author{Ron M.\ Adin}
\address{Department of Mathematics, Bar-Ilan University, Ramat-Gan 52900, Israel}
\email{radin@math.biu.ac.il}
\author{Eli Bagno}
\address{Jerusalem College of Technology, 21 Havaad Haleumi St., Jerusalem, Israel}
\email{bagnoe@g.jct.ac.il}
\author{Estrella Eisenberg}
\address{Jerusalem College of Technology, 21 Havaad Haleumi St., Jerusalem, Israel}
\email{estchoc@gmail.com}
\author{Shulamit Reches}
\address{Jerusalem College of Technology, 21 Havaad Haleumi St., Jerusalem, Israel}
\email{shulamit.reches@gmail.com}
\author{Moriah Sigron}
\address{Jerusalem College of Technology, 21 Havaad Haleumi St., Jerusalem, Israel}
\email{moria.sigron@mail.huji.ac.il}

\date{April 22, 2018}


\begin{abstract}

Gessel conjectured that the two-sided Eulerian polynomial, recording the common distribution of the descent number of a permutation and that of its inverse, has non-negative integer coefficients when expanded in terms of the gamma basis. This conjecture has been proved recently by Lin. 

We conjecture that an analogous statement holds for simple permutations, and use the substitution decomposition tree of a permutation (by repeated inflation) to show that this would imply the Gessel-Lin result. We provide supporting evidence for this stronger conjecture.



\end{abstract}


\maketitle

\section{Introduction}
\label{sec:intro}

Eulerian numbers enumerate permutations according to their descent numbers.  The  {\em two-sided Eulerian numbers}, studied by  Carlitz, Roselle, and Scoville \cite{CRS} constitute a natural generalization. These numbers count permutations according to their number of descents as well as the number of descents of the inverse permutation. 
%

Explicitly, 
the {\em descent set} of a permutation $\pi \in S_n$ is defined as:
\[
\Des(\pi) = \{i \in [n-1] \mid \pi(i)>\pi(i+1)\}.
\]
Denote $\des(\pi) = |\Des(\pi)|$ and $\ides(\pi) = \des(\pi^{-1})$, the {\em descent numbers} of $\pi$ and $\pi^{-1}$. 
For example, if $\pi= 2 4 6 1 3 5$ then $\Des(\pi)=\{3\}$, $\des(\pi)=1$, $\Des(\pi^{-1})=\{1,3,5\}$ and $\ides(\pi)=3$. 


A polynomial $f(q)$ is {\em palindromic} if its coefficients are the same when read from left to right as from right to left. 
Explicitly, if $f(q)=a_rq^r+a_{r+1}q^{r+1} +\cdots +a_sq^s$ with $a_r, a_s \ne 0$ and $r \le s$, then we require $a_{r+i}=a_{s-i}$ $(\forall i)$; equivalently,
$f(q) = q^{r+s} f(1/q)$.
Following Zeilberger \cite{Z}, we define the {\it darga} of $f(q)$ as above to be $r+s$; the zero polynomial is considered to be palindromic of each nonnegative darga. 
The set of palindromic polynomials of darga $n-1$ is a vector space of dimension $\lfloor (n+1)/2 \rfloor$, with {\em gamma basis}
\[
\{q^j (1+q)^{n-1-2j} \mid 0 \le j \leq \lfloor (n-1)/2 \rfloor \}.
\]
The (one-sided) {\em Eulerian polynomial} 
\[
A_n(q) = \sum_{\pi \in S_n} q^{\des(\pi)}
\]
is palindromic of darga $n-1$, and thus there are real numbers $\gamma_{n,j}$ such that 
\[
A_n(q) = \sum_{0 \leq j \leq \lfloor (n-1)/2 \rfloor} \gamma_{n,j} q^j(1+q)^{n-1-2j}.
\]
See \cite[pp.\ 72, 78]{Pet Book} for details. 
Foata and Sch\"utzenberger \cite{F.S} proved that the coefficients $\gamma_{n,j}$ are actually non-negative integers.
The result of Foata and Sch\"utzenberger was reproved combinatorially, using an action of the group $\mathbb{Z}_2^n$ on $S_n$ which leads to an interpretation of each coefficient $\gamma_{n,j}$ as the number of orbits of a certain type. This method, called ``valley hopping", is described in \cite{F.St, Branden}. A nice exposition appears in \cite{P}.

\medskip

Now let $A_n(s,t)$ be the {\em two-sided Eulerian polynomial}
\[
A_n(s,t) = \sum_{\pi \in S_n} s^{\des(\pi)} t^{\ides(\pi)}.
\]
%
It is well known (see, e.g., \cite[p.\ 167]{P}) that the bivariate polynomial $A_n(s,t)$ satisfies
\begin{equation}\label{eq.pal}
A_n(s,t) = (st)^{n-1} A_n(1/s, 1/t)
\end{equation}
as well as
\begin{equation}\label{eq.inv}
A_n(s,t) = A_n(t,s).
\end{equation}
In fact, (\ref{eq.pal}) follows from the bijection from $S_n$ onto itself taking a permutation to its reverse, while (\ref{eq.inv}) follows from the bijection taking each permutation to its inverse.

A bivariate polynomial satisfying Equations (\ref{eq.pal}) and (\ref{eq.inv}) will be called (bivariate) {\em palindromic of darga} $n-1$.
Note that if we arrange the coefficients of a bivariate palindromic polynomial in a matrix, then this matrix is symmetric with respect to both diagonals. 

\begin{exa}
The two-sided Eulerian polynomial for $S_4$ is:
\[
A_{4}(s,t)=1+10st+10(st)^2+(st)^3+st^2+s^2t.
\]
Its matrix of coefficients is 
\[ 
\left( \begin{array}{cccc}
1 & 0 & 0 & 0\\
0 & 10 & 1 & 0 \\
0 & 1 & 10 & 0 \\
0 & 0 & 0 & 1 
\end{array} \right),
\] 
and is clearly symmetric with respect to both diagonals.
\end{exa}

It can be proved (see \cite[p.\ 78]{Pet Book}) that the set of bivariate palindromic polynomials of darga $n-1$ is a vector space of dimension $\lfloor (n+1)/2 \rfloor \cdot \lfloor (n+2)/2 \rfloor$, with {\em bivariate gamma basis}
\[
\{(st)^i(s+t)^j(1+st)^{n-1-j-2i} \mid i,j \ge 0,\, 2i+j \le n-1 \}.
\]
A bivariate palindromic polynomial is called {\em gamma-positive} if all the coefficients in its expression in terms of the bivariate gamma basis are nonnegative.
Gessel (see \cite[Conjecture 10.2]{Branden}) conjectured that the two-sided Eulerian polynomial $A_n(s,t)$ is gamma-positive. 
This has recently been proved by Lin \cite{Lin}.
Explicitly:
\begin{thm}\label{t:Gessel} 
{\rm (Gessel's conjecture, Lin's theorem)}
For each $n \geq 1$ there exist nonnegative integers $\gamma_{n,i,j}$ $(i,j \ge 0,\, 2i+j \leq n-1)$ such that
\[
A_n(s,t)=\sum\limits_{i,j}{\gamma_{n,i,j}(st)^i(s+t)^j(1+st)^{n-1-j-2i}}.
\]
\end{thm}

An explicit recurrence for the coefficients $\gamma_{n,i,j}$ was described by Visontai~\cite{V}. This recurrence does not directly imply the positivity of the coefficients, but Lin~\cite{Lin} managed to use it to eventually prove Gessel's conjecture. Unlike the univariate case, no combinatorial proof of Gessel's conjecture is known.

\medskip

Simple permutations (for their definition see Section~\ref{sec:simple permutations}) serve as building blocks of all permutations. We propose here a strengthening of Gessel's conjecture, for the class of simple permutations.
\begin{con}
\label{conj:simple}
For each positive $n$, the bivariate polynomial
\[
\simp_{n}(s,t) = \sum_{\sigma \in \Simp_n} s^{\des(\sigma)} t^{\ides(\sigma)}
\]
is gamma-positive, where $\Simp_n$ is the set of simple permutations of length $n$.
\end{con}
Using the substitution decomposition tree of a permutation (by repeated inflation), we show how this cojecture implies the Gessel-Lin result. 
A combinatorial proof of the conjecture will give a combinatorial proof of the Gessel-Lin result.
We also provide supporting evidence for this stronger conjecture.


\medskip

The rest of the paper is organized as follows.
Section \ref{sec:simple permutations} contains background material concerning simple permutations, inflation, and the substitution decomposition tree of a permutation. 
In Section \ref{sec:partial settlement} we introduce combinatorial involutions on the tree, and use them to give a combinatorial proof of Gessel's conjecture for a certain class of permutations, $H(5) \cap S_n$. 
In Section~\ref{sec:reduction}
we show how, more generally, Lin's theorem (Gessel's conjecture) follows combinatorially from Conjecture \ref{conj:simple}. 
Finally, in Section \ref{sec:generating function}, we give a formula for $\simp_{n}(s,t)$ which may have independent value.

\section{Simple permutations and inflation}\label{sec:simple permutations}

We start by presenting some preliminaries concerning simple permutations, inflation and the substitution decomposition tree. Original papers will be mentioned occasionally, but terminology and notation will follow (with a few convenient exceptions) the recent survey \cite{Vatter}.

\begin{defn}
Let $\pi=a_1 \ldots a_n \in S_n$.   
A {\em block} (or {\em interval}) of $\pi$ is a nonempty contiguous
sequence of entries $a_i a_{i+1} \ldots a_{i+k}$ whose values also form a contiguous sequence of integers.
\end{defn}

\begin{exa}
If $\pi = 2647513$ then $6475$ is a block but $64751$ is not. 
\end{exa}

Each permutation can be decomposed into singleton blocks, and also forms a single block by itself; these are the {\em trivial blocks} of the permutation. All other blocks are called {\em proper}.

\begin{defn}
A permutation 
is {\em simple} if it has no proper blocks. 
\end{defn}

\begin{exa}\label{ex:simple up to 5}
The permutation $3517246$ is simple. 
\end{exa}

The simple permutations of length $n \le 2$ are $1$, $12$ and $21$.
There are no simple permutations of length $n=3$. 
Those of length $n=4$ are $2413$ and its inverse (which is also its reverse).
For length $n=5$ they are $24153$, $41352$, their reverses and their inverses (altogether $6$ permutations). 

\begin{defn}
A {\em block decomposition} of a permutation is a partition of it into disjoint blocks. 
\end{defn}

For example, the permutation $\sigma=67183524$  can be decomposed as $67\ 1\ 8\ 3524$. 
In this example, the relative order between the blocks forms the permutation $3142$, i.e., if we take for each block one of its digits as a representative then the set of representatives is order-isomorphic to $3142$. 
Moreover, the block $67$ is order-isomorphic to $12$, and the block $3524$ is order-isomorphic to $2413$. These are instances of the concept of {\em inflation}, defined as follows.

\begin{defn}
Let $n_1, \ldots, n_k$ be positive integers with $n_1 + \ldots + n_k = n$.
The {\em inflation} of a permutation $\pi \in S_k$ by permutations $\alpha_i \in S_{n_i}$ $(1 \leq i \leq k)$ is 
the permutation $\pi[\alpha_1, \ldots, \alpha_k] \in S_n$ obtained by replacing the $i$-th entry of $\pi$ by a block which is order-isomorphic to the permutation $\alpha_i$
on the numbers $\{s_i + 1, \ldots, s_i + n_i\}$ instead of $\{1, \ldots, n_i\}$, where $s_i = n_1 + \ldots + n_{i-1}$ $(1 \leq i \leq k)$. 
\end{defn}

\begin{exa}
The inflation of $2413$ by $213,21,132$ and $1$ is 
\[
2413[213,21,132,1]=546 \ 98 \ 132 \ 7.
\]
\end{exa}

A very important fact is that inflation is additive on both $\des$ and $\ides$. 

\begin{obs}\label{respect des}
Let $\sigma=\pi[\alpha_1,\dots,\alpha_k]$. Then
\[
\des(\sigma)=\des(\pi)+\sum_{i=1}^n{\des(\alpha_i)}
\]
and
\[
\ides(\sigma)=\ides(\pi)+\sum_{i=1}^n{\ides(\alpha_i)}.
\]
\end{obs}

Two special cases of inflation, deserving special attention, are the {\em direct sum} and {\em skew sum} operations, defined as follows.

\begin{defn}

Let $\pi \in S_m$ and $\sigma \in S_n$. 
The {\em direct sum} of $\pi$ and $\sigma$ is the permutation $\pi \oplus \sigma \in S_{m+n}$ defined by
\[
(\pi \oplus \sigma)_i =
\begin{cases} 
\pi_i, & \text{if } i \leq m; \\
\sigma_{i-m}+m, & \text{if } i>m,
\end{cases}
\]
and their {\em skew sum} is the permutation $\pi \ominus \sigma \in S_{m+n}$ defined by
\[
(\pi \ominus \sigma)_i = 
\begin{cases} 
\pi_{i}+n, & \text{if } i \leq m; \\
\sigma_{i-m},  & \text{if } i>m.
\end{cases}
\]
\end{defn}

\begin{exa}
If $\pi=132$ and $\sigma=4231$ then $\pi \oplus \sigma=1327564$ and $\pi \ominus \sigma=5764231$
\end{exa}

Note that $\pi \oplus \sigma = 12[\pi,\sigma]$ and $\pi \ominus \sigma = 21[\pi,\sigma]$. 

\begin{defn}\label{sum-indecomposable}
A permutation is {\em sum-indecomposable} (respectively, {\em skew-indecomposable}) if it cannot be written as a direct (respectively, skew) sum. 
\end{defn}

The following proposition shows that every permutation has a canonical representation as an inflation of a simple permutation. 

\begin{prop}\label{t.substitution}
\cite[Theorem 1]{AAK}\cite[Proposition 3.10]{Vatter}
Let $\sigma \in S_n$ $(n \geq 2)$. Then there exist a unique integer $k \ge 2$, a unique simple permutation $\pi \in S_k$, and a sequence
of permutations $\alpha_1,\ldots,\alpha_k$ such that 
\[
\sigma = \pi[\alpha_1,\dots,\alpha_k].
\]

If $\pi \notin \{12,21\}$ then $\alpha_1, \ldots,\alpha_k$ are also unique.  

If $\pi=12$ $(\pi=21)$ then $\alpha_1, \alpha_2$ are unique as long as we require, in addition, that $\alpha_2$ is sum-indecomposable (respectively, skew-indecomposable). 
\end{prop}

\begin{exa}
The permutation $\sigma=452398167$ can be written as an inflation of the simple permutation $2413$:
\[
\sigma=2413[3412, 21, 1, 12].
\]
\end{exa}

\begin{rem} \label{right chain condition in permutations}
The additional requirements for $\pi = 12$ and $\pi = 21$ are needed for uniqueness of the expression. 
To see that, note that the permutation $123$ can be written as $12[12,1]=12 \ 3$ but also as $12[1,12]=1 \ 23$. The first expression is the one preferred above (with $\alpha_2$ sum-indecomposable). 
\end{rem}

One can continue the process of decomposition by inflation for the constituent permutations $\alpha_i$, recursively, until all the resulting permutations have length $1$. 
In the example above, $3412$ can be further decomposed as $3412=21[12,12]$, so that
\[
\sigma= 2413[21[12,12],21,1,12]
\]
and, eventually,
\[
\sigma= 2413[21[12[1,1],12[1,1]],21[1,1],1,12[1,1]].
\]
%
%

This information can be encoded by a tree, as follows.
\begin{defn}\label{tree of perm}
Represent each permutation $\sigma$ by a corresponding {\em substitution decomposition tree} $T_{\sigma}$, recursively, as follows. 
\begin{itemize}
\item
If $\sigma = 1 \in S_1$, represent it by a tree with one node.
\item
Otherwise, write $\sigma=\pi[\alpha_1,\ldots,\alpha_k]$ as in Proposition~\ref{t.substitution}, and represent $\sigma$ by a tree with a root node, labeled $\pi$, having $k$ ordered children corresponding to $\alpha_1,\ldots,\alpha_k$. Replace each child $\alpha_i$ by the corresponding tree $T_{\alpha_i}$. 
\end{itemize}
\end{defn}



\begin{figure}
\begin{center}
\includegraphics[scale=1]{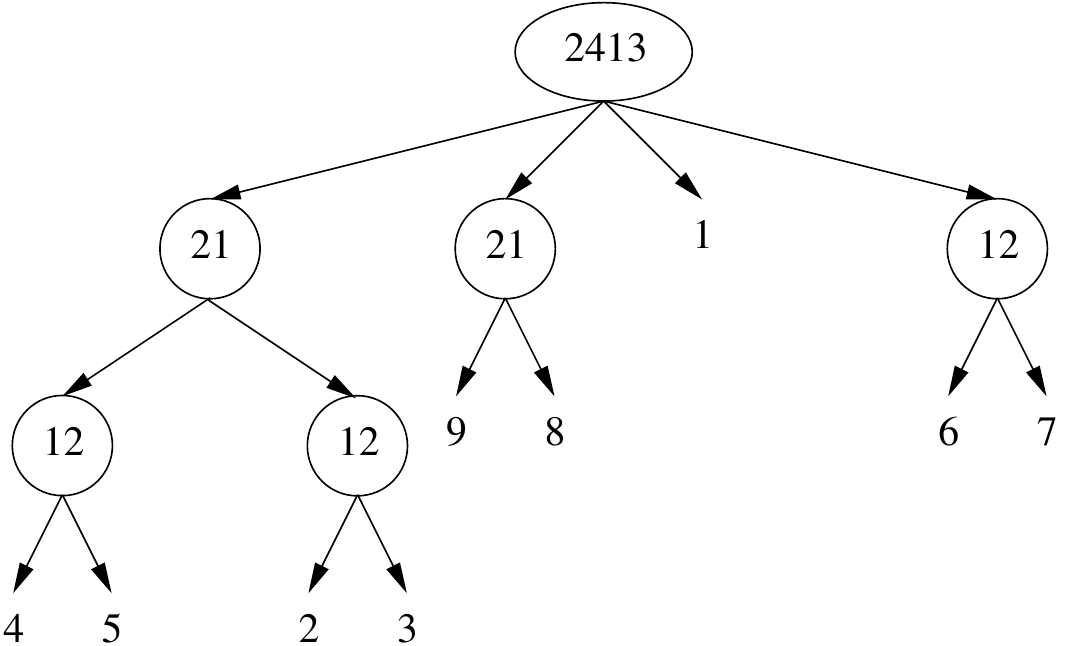}
\caption{The tree $T_\sigma$ for $\sigma=452398167$}
\label{fig:one tree}
\end{center}
\end{figure}

\begin{exa}
Figure \ref{fig:one tree} depicts the substitution decomposition tree $T_{\sigma}$ for $\sigma=452398167$. For clarity, the leaves are labeled by the corresponding values of the permutation $\sigma$, instead of simply $1$.
\end{exa}












Inflation can be extended to
sets of permutations (an operation called {\em wreath product} in \cite{A.S}).

\begin{defn}
Let $\mathcal{A}$ and $\mathcal{B}$ be sets of permutations. 
Define
\[
\mathcal{A}[\mathcal{B}] 
= \{\alpha[\beta_1,\ldots,\beta_k] \mid \alpha \in \mathcal{A},\, \beta_1,\ldots,\beta_k \in \mathcal{B} \}.
\]
\end{defn}

\begin{exa}
Let $A=\{12\}$ and $B=\{21,132\}$. 
Then 
\[
A[B]
= \{2143,21354,13254,132465\}.
\]
\end{exa}

\begin{defn}
A set $C$ of permutations is 
{\em substitution-closed} if $C = C[C]$.
The {\em substitution closure} $\langle C \rangle$ 
of a set $C$ of permutations is the smallest 
substitution-closed
set of permutations which contains $C$. 
\end{defn}

The inflation
operation is associative. Defining $C_1 = C$ and 
$C_{n+1} = C[C_n]$, we clearly have 
\[
\langle C \rangle = \bigcup_{n=1}^{\infty} C_n.
\]



\begin{defn}\label{def:Simp}
For a positive integer $n$, 
let $\Simp_n$ ($\Simp_{\leq n}$) be the set of all simple permutations of length $n$ (respectively, of length at most $n$).
Let $H(n)=\langle \Simp_{\leq n} \rangle$, the substitution closure of $\Simp_{\leq n}$.
\end{defn}

\begin{exa}
$H(2) = \langle \Simp_{\le 2} \rangle = \langle \{1,12,21\} \rangle$ 
is the set of all permutations that can be obtained from the trivial permutation $1$ by direct sums and skew sums. These are exactly the {\em separable permutations}, counted by the {\em large Schr\"oder numbers;} see \cite{W}. 
Separable permutations can also be described via pattern avoidance, namely 
\[
H(2)=\Av(3142,2413).
\]
For more details see \cite[following Proposition 3.2]{B.H.V.}. 
\end{exa}

\section{Gamma-positivity for $H(5)$}
\label{sec:partial settlement}

In this section we present a combinatorial proof of Gessel's conjecture (Lin's theorem) for the subset $H(5) \cap S_n$ of $S_n$ (for any positive $n$).

Fu, Lin and Zeng \cite{F.L.Z.} proved the following (univariate) gamma-positivity result.

\begin{prop} 
For each $n$ there exist nonnegative integers $\gamma_{n,k}$ $(0 \leq k \leq \lfloor (n-1)/2 \rfloor)$ such that
\[
\sum_{\pi \in H(2)\cap S_n}{t^{\des(\pi)} = \sum_{k=0}^{\lfloor (n-1)/2 \rfloor}} { \gamma_{n,k} t^k (1+t)^{n-1-2k}}
\]
\end{prop} 

By Observation~\ref{respect des},
if $\pi \in H(2)$ then $\des(\pi)=\ides(\pi)$. 
Hence, one can conclude the following restricted version of Gessel's conjecture for the set of separable permutations.

\begin{thm} \label{H(2)}
For each $n$ there exist non-negative integers $\gamma_{n,k}$ $(0 \leq k \leq \lfloor (n-1)/2 \rfloor)$ such that:
\[
\sum_{\pi \in H(2) \cap S_n}{s^{\des(\pi)}t^{\ides(\pi)}=\sum_{\pi \in H(2) \cap S_n}{(st)^{\des(\pi)}=\sum_{k=0}^{\lfloor (n-1)/2 \rfloor}}{ \gamma_{n,k} (st)^k (1+st)^{n-1-2k}}}.
\]
\end{thm} 


In order to extend Theorem \ref{H(2)} further, let us introduce some more definitions.

\begin{defn}\label{def:BRC}
Let $T$ be a tree with all internal nodes labeled by simple permutations. 
A {\em binary right chain} (BRC) is a maximal nonempty chain composed of consecutive right descendants, all of which are from the set $\{12,21\}$.
The {\em length} of a BRC is the number of nodes in it.
Denote by $r_{odd}(T)$ the number of BRC of odd length in $T$.
\end{defn}

\begin{exa}
The tree $T_{\sigma}$ in Figure \ref{fig:one tree} has $4$ BRC, and $r_{odd}(T_{\sigma})=3$.
\end{exa}

\begin{figure}
\begin{center}
\includegraphics[scale=0.8]{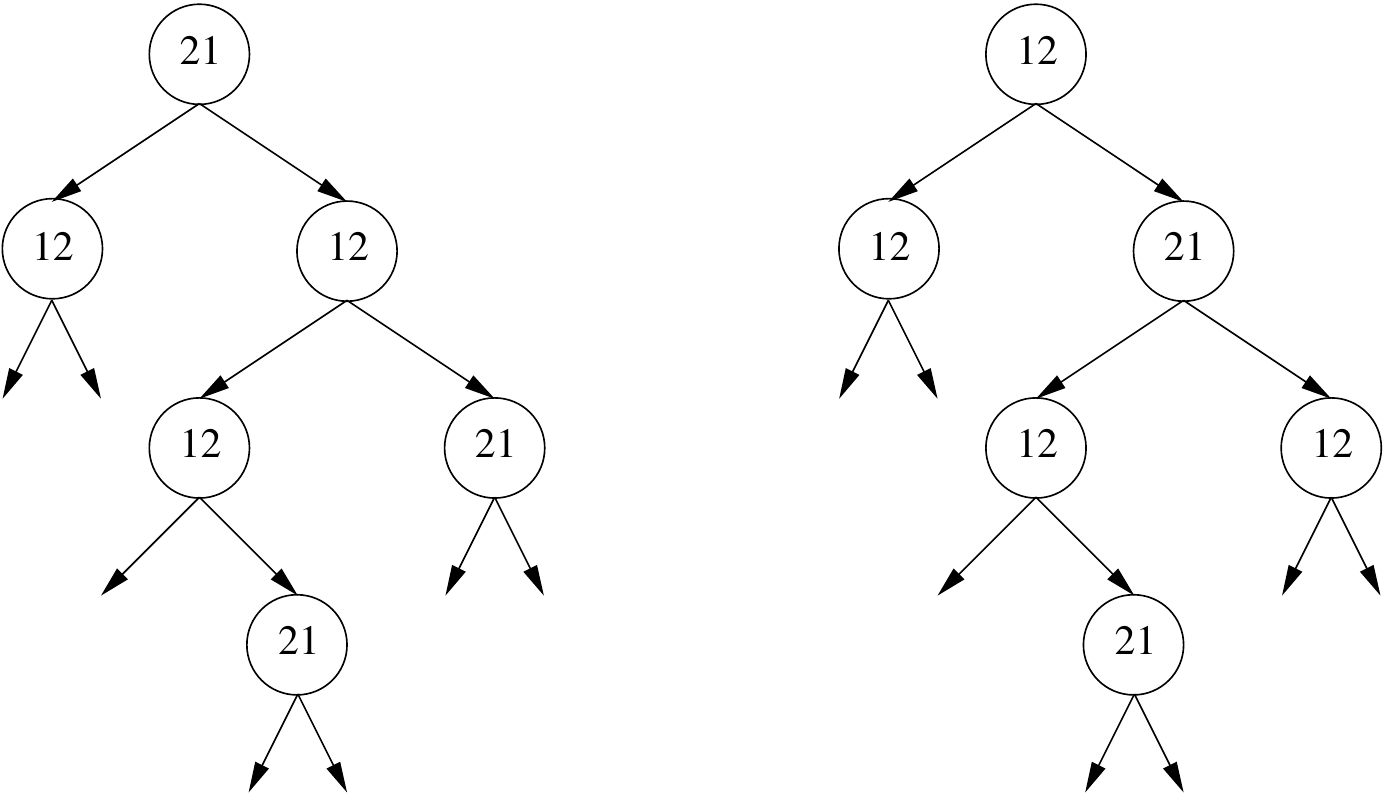}
\caption{Left: the tree $T$. Right: The tree $\phi_1(T)$.}
\label{fig:two trees}
\end{center}
\end{figure}

\begin{defn}\label{def:G-tree}
A tree $T$ is called a {\em G-tree} if 
it satisfies:
\begin{enumerate}
\item
Each leaf is labeled by $1$.
\item
Each internal node is labeled by a simple permutation ($\ne 1$), and the number of its children is equal to the length of the permutation.
\item
The labels in each BRC alternate between 12 and 21.
\end{enumerate}
Denote by $\mathcal{GT}_n$ the set of all G-trees with $n$ leaves. 
\end{defn}



\begin{lem}\label{bijection perms to trees}
The map $f_n : S_n \to \mathcal{GT}_n$ sending each permutation $\sigma$ to its substitution decomposition tree $T_{\sigma}$, as in Definition~\ref{tree of perm}, is a bijection.
\end{lem} 

\begin{proof}
Follows immediately from Proposition~\ref{t.substitution}.
The last condition in Definition~\ref{def:G-tree} reflects the extra restrictions for the cases $\pi \in \{12, 21\}$ in Proposition~\ref{t.substitution}.
\end{proof}


Let $T=T_{\pi}$ be a G-tree, and let $\{C_i \mid 1 \leq i \leq r_{odd}(T)\}$ be the set of all BRC of odd length in $T$. For each $i$, let $\phi_i(T)$ be the tree obtained from $T$ by switching $12$ and $21$ in each of the nodes of $C_i$. (A similar action was introduced in \cite{F.L.Z.} for univariate polynomials.)
Clearly, each operator $\phi_i$ is an involution, and the various $\phi_i$ commute.
By Observation~\ref{respect des}, each $\phi_i$ changes both $\des(\pi)$ and $\ides(\pi)$ by $\pm1$.

\begin{exa}
Consider $\pi = 6713254$. The corresponding tree $T=T_{\pi}$ appears on the left side of Figure \ref{fig:two trees}, and has $r_{odd}(T)=2$. If 
$C_1$ is the unique BRC of length $3$ in $T$,
then $\phi_1(T)$ is the tree on the right side of the figure. The permutation corresponding to $\phi_1(T)$ is $1257634$. Note that $\phi_1$ decreased both $\des(\pi)$ and $\ides(\pi)$ by $1$. 
\end{exa}

Let $T = T_{\pi}$ be a G-tree, and let $l_1,\ldots,l_k$ be the labels of nodes in $T$ that belong to the set $\Simp_4 = \{2413,3142\}$. 
Define $\psi_j(T)$ $\textup{(}1 \leq j \leq k\textup{)}$ to be the tree obtained from $T$ by switching the label $l_j$ from $2413$ to $3142$, or vice versa. 
Again, it is easy to see that the $\psi_j$ are commuting involutions, and each $\psi_j$ commutes with each $\phi_i$.
Switching from $2413$ to $3142$ increases $\des(\pi)$ by $1$ while decreasing $\ides(\pi)$ by $1$.

\begin{rem}
Each of the $6$ simple permutations $\pi \in \Simp_5$ has $\des(\pi) = \ides(\pi) = 2$, and we don't need to define involutions for them.
\end{rem}

\begin{defn}
For any two G-trees $T_1$ and $T_2$, write $T_1 \sim T_2$ if $T_2$ can be obtained from $T_1$ by a sequence of applications of the involutions $\phi_i$ and $\psi_j$.
\end{defn}

Clearly $\sim$ is an equivalence relation, partitioning the set $\mathcal{GT}_n$ (equivalently, the group $S_n$) into equivalence classes. 

\begin{defn}
For each equivalence class in $\mathcal{GT}_n$, let $T_0$ be the unique tree in this class in which each odd BRC begins with $12$ and each node representing a simple permutation of length $4$ is labeled $2413$. 
The corresponding permutation $\pi_0$ has the minimal number of descents in its class. 
The tree $T_0$ and the permutation $\pi_0$ are called the {\em minimal representatives} of their equivalence class.
\end{defn}




 



\begin{lem}\label{over class}
Let $A$ be an equivalence class of permutations in $H(5)\cap S_n$.
There exist nonnegative integers $i$ and $j$ such that
\[
\sum_{\sigma \in A} s^{\des(\sigma)} t^{\ides(\sigma)} 
= (st)^{i} (s+t)^{j} (1+st)^{n-1-2i-j}.
\]
\end{lem}

\begin{proof}
Let $\pi_0$ be the minimal representative of $A$, and let $T_0 = T_{\pi_0}$. 
For $i \in \{2,4,5\}$, let $v_i$ be the number of nodes of $T_0$ having labels of length $i$.  
Since $T_0$ has exactly $n$ leaves, its total number of nodes (including leaves) is 
\[
v = n + v_2 + v_4 + v_5.
\]
%
On the other hand, counting the children of each node gives
\[
v - 1 = \sum_{i\in \{2, 4, 5\}} i v_i.
\]
It follows that
\[
n-1 = v_2 + 3v_4 + 4v_5.
\]


Let $r = r_{odd}(T_0)$ be the number of odd BRC in $T_0$, and 
let $d_2$ be the number of nodes labeled $21$.
By definition, each BRC alternates between $12$ and $21$ and each odd BRC in $T_0$ starts with $12$. 
It follows that
\[
v_2 = 2d_2 + r,
\]
so that
\begin{equation}\label{eq:number_of_nodes}
n-1 = r + 2d_2 + 3v_4 + 4v_5.
\end{equation}

Now recall that 
$\des(12)=\ides(21)=0$,
$\des(21)=\ides(21)=1$,
$\des(2413)=1$, $\ides(2413)=2$, 
$\des(3142)=2$, $\ides(3142)=1$, 
and for each simple permutation $\sigma$ of length $5$ we have $\des(\sigma)=\ides(\sigma)=2$. 
it follows from Observation~\ref{respect des} that
\[
\des(\pi_0)=d_2+v_4+2v_5
\]
and
\[
\ides(\pi_0)=d_2+2v_4+2v_5,
\]
so that the bivariate monomial corresponding to $\pi_0$ is
\[
s^{\des(\pi_0)} t^{\ides(\pi_0)}
= (st)^{d_2 + v_4 + 2v_5} t^{v_4}.
\]

Consider now the whole equivalence class $A$, whose elements are obtained from $\pi_0$ by applications of the commuting involutions $\phi_1, \ldots, \phi_r$ and $\psi_1, \ldots, \psi_k$, 
where $r = r_{odd}(T_0)$ is the number of BRC in $T_0$ and $k = v_4$.
Each application of $\phi_i$ multiplies the monomial by $st$, and each application of $\psi_j$ multiplies it by $st^{-1}$. It follows that



\[
\begin{aligned}
\sum_{\sigma \in A}
s^{\des(\sigma)} t^{\ides(\sigma)}
&= (st)^{d_2+v_4+2v_5} t^{v_4} (1+st)^{r} (1 + st^{-1})^{v_4} \\
&= (st)^{d_2+v_4+2v_5} (s+t)^{v_4} (1+st)^{r}.
\end{aligned}
\]
Denoting $i = d_2+v_4+2v_5$ and $j = v_4$ will complete the proof, once we show that
\[
2(d_2+v_4+2v_5) + v_4 +r = n-1;
\]
but this follows immediately from equation~\eqref{eq:number_of_nodes} above.






\end{proof}

Lemma \ref{over class} immediately implies one of the main results of this paper:  

\begin{thm}\label{main theorem}
For each $n \ge 1$, the polynomial
\[
\sum_{\sigma \in H(5) \cap S_n} s^{des(\sigma)} t^{ides(\sigma)}
\]
is gamma-positive.
\end{thm}

\section{Gamma-positivity for simple permutations}
\label{sec:reduction}

For each positive integer $n$, the set $\Simp_n$ of simple permutations of length $n$ is invariant under taking inverses and reverses. It follows that the bivariate polynomial 
\[
\simp_{n}(s,t) = \sum_{\sigma \in \Simp_n} s^{\des(\sigma)} t^{\ides(\sigma)}
\]
is palindromic, and can be expanded in the gamma basis. 
Conjecture~\ref{conj:simple}, presented in the Introduction, states that this polynomial is, in fact, gamma-positive.
The main result of this section is the following.

\begin{thm}
Conjecture~\ref{conj:simple} implies Gessel's conjecture (Lin's theorem), Theorem~\ref{t:Gessel}.
\end{thm}

\begin{proof}
For each permutation $\sigma$ of length $n \ge 2$, consider its substitution decomposition tree $T_{\sigma}$. Each internal node of $T_{\pi}$ is labeled by some simple permutation $\pi$ of length $\ell = \ell(\pi) \ge 2$. Replace $\pi$ by $\ell$, to obtain a {\em simplified tree} $T'_{\sigma}$ (with internal nodes labeled by numbers).
For permutations $\sigma_1, \sigma_2 \in S_n$ define $\sigma_1 \sim \sigma_2$ if $T'_{\sigma_1} = T'_{\sigma_2}$.
Clearly $\sim$ is an equivalence relation on $S_n$, with each equivalence class corresponding to a unique simplified tree $T'$. Denote such a class by $A(T')$.

Define a BRC of $T'$ (in analogy to Definition~\ref{def:BRC}) to be a maximal nonempty chain of consecutive right descendants, all labeled $2$.
How can we recover a permutation $\sigma \in A(T')$ from the tree $T'$?
Each internal node, labeled by a number $\ell$, can be relabeled by any simple permutation of length $\ell$,
with the single restriction that the labels in each BRC must alternate between $12$ and $21$, starting with either of them.
It thus follows, by Observation~\ref{respect des}, that for each simplified tree $T'$, the polynomial 
\[
\sum_{\sigma \in A(T')} s^{\des(\sigma)} t^{\ides(\sigma)}
\]
is a product of factors, as follows: 
\begin{itemize}
\item 
Each internal node with label $\ell \geq 4$ contributes a factor $\simp_\ell(s,t)$. 	
\item
Each BRC of even length $2k$  contributes a factor $2(st)^k$.
\item
Each BRC of odd length $2k+1$  contributes a factor $(st)^k (1+st)$.
\end{itemize}

By Conjecture~\ref{conj:simple}, all those factors are gamma-positive, and so is their product. Summing over all equivalence classes in $S_n$ completes the proof.
\end{proof}

It is clear from the arguments above that a {\em combinatorial} proof of Conjecture~\ref{conj:simple} will immediately yield a combinatorial proof of Theorem~\ref{t:Gessel}.
In fact, the preceding section contains such a combinatorial proof assuming there are only labels $\ell \le 5$,
using $\simp_4(s,t) = st (s+t)$ and $\simp_5(s,t) = 6(st)^2$.
We were unable to extend the combinatorial arguments to length $6$, although the corresponding polynomial is indeed gamma-positive:
\[
\simp_6(s,t) = st(s+t)^2(1+st)+5(st)^2(1+st)+14(st)^2(s+t).
\]
In fact, 
Conjecture~\ref{conj:simple} has been verified by computer for all $n \le 12$.
\section{The bi-Eulerian polynomial for simple permutations}\label{sec:generating function}

In \cite{AAK}, the ordinary generating function for the {\em number} of simple permutations was shown to be very close to the functional inverse of the corresponding generating function for all permutations. In this section we refine this result by considering also the parameters $\des$ and $\ides$, thus obtaining a formula for $\simp_n(s,t)$.

\medskip

Recall from Definition \ref{sum-indecomposable} the notions of sum-indecomposable and skew-indecomposable permutations.

\begin{defn}
For each positive integer $n$, denote by $I_n^+$ (respectively, $I_n^-$) the set of all sum-indecomposable (respectively, skew-indecomposable) permutations in $S_n$.
\end{defn}

\begin{defn}
Let 

\begin{align*}
F(x,s,t) := \sum_{n=1}^{\infty} 
\left( \sum_{\pi \in S_n} s^{\des(\pi)} t^{\ides(\pi)} \right) x^n, \\
I^+(x,s,t) := \sum_{n=1}^{\infty} 
\left( \sum_{\pi \in I^+_n} s^{\des(\pi)} t^{\ides(\pi)} \right) x^n, \\
I^-(x,s,t) := \sum_{n=1}^{\infty} 
\left( \sum_{\pi \in I^-_n} s^{\des(\pi)} t^{\ides(\pi)} \right) x^n, \\
S(x,s,t) := \sum_{n=4}^{\infty} 
\left( \sum_{\pi \in \Simp_n} s^{\des(\pi)} t^{\ides(\pi)} \right) x^n.
\end{align*}

\end{defn}

Note that the summation in the definition of $S(x,s,t)$ is only over $n \ge 4$. We want to find relations between these generating functions.

From now on, we consider $F(x,s,t)$ etc.\ as formal power series in $x$, with coefficients in the field of rational functions $\mathbb{Q}(s,t)$. We therefore use the short notation $F(x)$, or even $F$.
For example, the composition $S \circ F$ means that $F$ is substituted as the $x$ variable of $S(x,s,t)$.
By Proposition~\ref{t.substitution} and Observation~\ref{respect des},

\begin{align*}
F 
&= x + I^+ F + st I^- F + \sum_{n=4}^{\infty} \simp_n F^n \\
&= x + I^+ F + st I^- F + S \circ F 
\end{align*}
and similarly
\[
I^+ = x + st I^- F + S \circ F
\]
and
\[
I^- = x + I^+ F + S \circ F.
\]
Rearranging, we have
\begin{align*}
F I^+ + stF I^- + (S \circ F + x) &= F \\
-I^+ + stF I^- + (S \circ F + x) &= 0 \\
F I^+ - I^- + (S \circ F + x) &= 0
\end{align*}

This is a system of linear equations in $I^+$, $I^-$ and $S \circ F + x$. Its unique solution is

\begin{equation}\label{e.solutions}
\begin{aligned}
I^+ &= \frac{F}{1 + F} \\
I^- &= \frac{F}{1 + stF} \\
S \circ F + x &= \frac{F(1 - stF^2)}{(1 + F)(1 + stF)}
\end{aligned}
\end{equation}
Note that the reversal map $\pi \mapsto \pi'$, defined by $\pi'(i)=n-1-\pi(i)$ $(1 \le i \le n)$, is a bijection from $S_n$ onto itself 
(and also from $I_n^+$ onto $I_n^-$),
satisfying $\des(\pi') = n-1-\des(\pi)$ and $\ides(\pi') = n-1-\ides(\pi)$. Therefore:
\[
\begin{aligned}
F(x,s,t) &= \frac{1}{st} F(xst, 1/s, 1/t) \\
I^-(x,s,t) &= \frac{1}{st} I^+(xst, 1/s, 1/t).
\end{aligned}
\]
This agrees with the first two equations in \eqref{e.solutions}.
Denoting $u=F(x)$, the third equation in \eqref{e.solutions} gives an explicit expression for $S(u,s,t)$:

\[
S(u,s,t) = -F^{\langle -1 \rangle}(u) + \frac{u(1-stu^2)}{(1+u)(1+stu)}\,,
\]
where $x = F^{\langle -1 \rangle}(u)$ is the functional inverse of $u = F(x)$.
Further manipulations with partial fractions give the following.
\begin{prop}
\begin{equation}\label{eq.Sust}
S(u,s,t) = 
-F^{\langle -1 \rangle}(u) + \frac{u}{1 + stu} + \frac{u}{1 + u} -u.
\end{equation}
\end{prop}

Using the expansions
\[
S(u,s,t) = \sum_{n \ge 4} \simp_n(s,t) u^n
\]
and
\[ 
F^{\langle -1 \rangle}(u) = \sum_{n \ge 1} f_n^{\langle -1 \rangle}(s,t) u^n,
\]
we finally obtain a formula for $\simp_n(s,t)$.
\begin{cor}
\[
\simp_n(s,t) = -f_n^{\langle -1 \rangle}(s,t) + (-1)^{n-1} + (-st)^{n-1}
\qquad (n \ge 4).
\]
\end{cor}

%

\section*{Acknowledgments}
The authors thank Mathilde Bouvel and an anonymous referee for useful comments.
RMA thanks the Institute for Advanced Studies, Jerusalem, for its hospitality during part of the work on this paper.
Work of RMA was partially supported bu an MIT-Israel MISTI grant.

\end{document}